\documentclass[11pt]{article}
\usepackage{amssymb}
\usepackage{hyperref}
\hypersetup{colorlinks=true,linkcolor=red,citecolor=green,filecolor=magneta,urlcolor=cyan}
\usepackage{amsthm}
\usepackage{amsmath}
\usepackage{graphicx}
 \newtheorem{thm}{Theorem}[section]
 \newtheorem{cor}[thm]{Corollary}
 \newtheorem{lem}[thm]{Lemma}
 \newtheorem{prop}[thm]{Proposition}
 \theoremstyle{definition}
 \newtheorem{defn}[thm]{Definition}
 \theoremstyle{remark}
 \newtheorem{rem}[thm]{Remark}
 
 \numberwithin{equation}{section}

\begin{document}

%
%
%
%
%
%
%
%
%

\title{Exponential quadrature rules \\for linear fractional differential equations}

\author{Roberto Garrappa \\
\small Universit\`a degli Studi di Bari - Dipartimento di Matematica\\ 
\small Via E. Orabona 4 - 70125 Bari - Italy \\
\texttt{roberto.garrappa@uniba.it}
\and
Marina Popolizio \\
\small Universit\`a del Salento - Dipartimento di Matematica e Fisica ``Ennio De Giorgi"\\ 
\small Via per Arnesano - 73100 Lecce - Italy \\
\texttt{marina.popolizio@unisalento.it}
}

\date{October 21, 2013 
}

\maketitle

\begin{abstract}
This paper focuses on the numerical solution of initial value problems for fractional differential equations of linear type. 
The approach we propose grounds on expressing the solution in terms of some integral weighted by a generalized Mittag--Leffler function. Then suitable quadrature rules  are devised and order conditions of algebraic type are derived.
Theoretical findings are validated by means of numerical experiments and the effectiveness of the proposed approach is illustrated by means of comparisons with other standard methods.
\end{abstract}





\def\Rset{{\mathbb{R}}}
\def\Nset{{\mathbb{N}}}
\def\Cset{{\mathbb{C}}}

\section{Introduction}

Fractional calculus is nowadays receiving a great attention from scientists of different branches; models involving derivatives of non integer order are indeed successfully employed to describe many complex processes and real--world
phenomena in several areas such as biology, engineering, finance, physics and so forth (e.g, see \cite{BuenoOrovioKayGrauRodriguezBurrage2013,Cafagna2007,DancaGarrappaTangChen2013,LinoMaione2013,Mainardi2010,RazminiaBaleanu2013}). 


At the same time the subject of the numerical solution of  fractional differential equations (FDEs) has been extensively addressed in order to provide efficient tools for the simulation of fractional order models and new important contributions have been recently published;
we cite, among the others, \cite{DiethelmFreed1999,GarrappaPopolizio2011_CAMWA,Garrappa2013_CAMWA,LiChenYe2011,Lubich1986,Moret2013} and references therein. Recently, in \cite{Garrappa2014_MCS} the investigation of numerical methods has been extended also to FDEs with discontinuous right--hand side \cite{DieciLopez2009,DieciLopez2011,DieciLopez2012}. 

A typical impasse in the numerical solution of FDEs is the long and persistent memory, a common problem with methods for Volterra Integral Equations \cite{CardoneConte2013,CardoneIxaruPaternoster2010}.

However one of the major difficulties in the development of numerical methods for FDEs is to obtain a sufficiently high accuracy. Classical approaches for ordinary differential equations (ODEs) based on polynomial replacements usually fail to provide satisfactory results when applied to FDEs: indeed, as shown in \cite{Lubich1983}, the asymptotic expansion of the solution of a FDE possesses real powers which can not be suitably approximated by means of polynomials.


Due to this lack of smoothness, methods like product integration (PI) rules (which can be considered as the counterpart of Adams multistep methods for ODEs) suffer from a severe order barrier when applied in a fractional order context; it has been shown \cite{Dixon1985} that, regardless of the degree of the polynomial replacement, an order of convergence not exceeding 2 is expected for FDEs in the more general situation (see also \cite{DiethelmFordFreed2004,LiTao2009}).

The main aim of this paper is to investigate alternative approaches in which the special nature of the problem can be exploited in order to device suitable numerical methods capable to provide accuracy of high order. This is the case of linear FDEs, whose generic expression is 
\begin{equation}\label{eq:FDE_Linear}
		D^{\alpha}_{t_0} y(t) + \lambda y(t) = f(t) , 
\end{equation}
where $\alpha\in\Rset$ is the positive fractional order, $\lambda\in\Rset$, $y(t):[t_0,T] \to \Rset$ and the forcing term $f(t)$ is assumed sufficiently smooth. The symbol $D^{\alpha}_{t_0}$ denotes the derivative operator of non--integer order $\alpha$, with respect to the origin $t_0$, according to the Caputo definition \cite{Diethelm2010,KilbasSrivastavaTrujillo2006,Podlubny1999} 
\[
	D^{\alpha}_{t_{0}} y(t) =
	\frac{1}{\Gamma(m-\alpha)} 
	  \int_{t_{0}}^t \frac{y^{(m)}(u)}{\bigl(t-u\bigr)^{\alpha+1-m}} du ,
\]
with $\Gamma(\cdot)$ the Euler gamma function, $m = \left\lceil \alpha \right\rceil$ the smallest integer such that $\alpha<m$ and $y^{(m)}$ the standard derivative of integer order; equation (\ref{eq:FDE_Linear}) is usually coupled with initial conditions in the form 
\begin{equation}\label{eq:_ic_FDE_Linear}
	y^{(k)}(t_{0}) = y_{0,k}
	, \quad
	k=0,\dots,m-1 .
\end{equation}

The scalar equation (\ref{eq:FDE_Linear}) represents a model problem having importance in several areas, for instance in control theory. Anyway, more involved models, such as multidimensional systems (possibly coming from spatial discretization of partial differential equations with time and/or space fractional derivatives), can be considered and analyzed starting from (\ref{eq:FDE_Linear}). 

The basic strategy underlying the approach proposed in this paper consists in reformulating (\ref{eq:FDE_Linear}) in terms of some special weighted integrals having just the forcing term $f(t)$ as the integrand and hence device specific quadrature rules for their numerical approximation. In this way we expect high accuracy under suitable and reasonable assumptions of smoothness just for the integrand $f(t)$ (a quite common situation in real--life applications); the same goal can be hardly achieved by using a classical formulation in which quadrature rules are applied to the usually non--smooth whole vector field $\lambda y(t) + f(t)$.

Rules of this kind can be considered as a generalization of the {\emph{exponential quadrature rules}} employed in the numerical treatment of linear ODEs \cite{HochbruckOstermann2010} to improve stability properties: as we will see, also convergence properties can benefit from approaches of this kind when applied to FDEs.



This paper is organized as follows. In Section \ref{S:SpecialFunction} we present the alternative formulation for (\ref{eq:FDE_Linear}) in terms of some generalized Mittag--Leffler functions and we present a brief review of the main properties of these important functions. In Section \ref{S:NumericalMethod} we introduce the new quadrature rules, we provide a systematic theory concerning their order of accuracy and we derive the conditions for generating rules of any prescribed degree of precision. By using these results, we perform in Section \ref{S:ApplicationsFDS} the error analysis for the application to the numerical solution of FDEs. Section \ref{S:EvaluationQuadratureWeights} is hence devoted to the evaluation of the quadrature weights by means of some techniques combining the inversion of the Laplace transform with rational approximations; the multidimensional case is then addressed. Finally in Section \ref{S:NumericalExperiments} we present some numerical tests to validate theoretical results and compare the approach investigated in this paper with classical methods. 


\section{Variation--of--constant formula and generalized Mittag--Leffler functions}\label{S:SpecialFunction}

By means of the Laplace transform, it is immediate to see that problem (\ref{eq:FDE_Linear}-\ref{eq:_ic_FDE_Linear}) can be rewritten in the spectral domain as
\[
	s^{\alpha} Y(s) - \sum_{k=0}^{m-1} s^{\alpha-k-1} y_{0,k} + \lambda Y(s) = F(s) ,
\]
where $Y(s)$ and $F(s)$ are the Laplace transform of $y(t)$ and $f(t)$ respectively \cite{Podlubny1999}; hence $Y(s)$ can be obtained as 
\[
	Y(s) = \sum_{k=0}^{m-1} \frac{s^{\alpha-k-1}}{s^{\alpha}+\lambda} y_{0,k} + \frac{F(s)}{s^{\alpha}+\lambda} .
\]

Thus, turning back to the temporal domain, the solution of (\ref{eq:FDE_Linear}-\ref{eq:_ic_FDE_Linear}) can be written as
\begin{equation}\label{eq:FDE_VariationConstantFormula}
	y(t) =
	\sum_{k=0}^{m-1}  e_{\alpha,k+1}(t-t_{0};\lambda) y_{0,k} + \int_{t_{0}}^{t} e_{\alpha,\alpha}(t-s;\lambda) f(s) ds,
\end{equation}
where $e_{\alpha,\beta}(t;\lambda)$ is the inverse Laplace transform  of $s^{\alpha-\beta}/\bigl( s^{\alpha}+\lambda \bigr)$; it is a generalization of the Mittag--Leffler (ML) function
\[
	E_{\alpha,\beta}(z) = \sum_{k=0}^{\infty} \frac{z^{k}}{\Gamma(\alpha k + \beta)},\,\,z\in\Cset,
\]
and can be evaluated \cite{Mainardi2010,Podlubny1999} according to
\[
	e_{\alpha,\beta}(t;\lambda) = t^{\beta-1} E_{\alpha,\beta}(-t^{\alpha} \lambda) .
\]

For any $z\in\Cset$ the behavior of $e_{\alpha,\beta}(t;z)$ for $t$ close to the origin can be studied directly from its definition; indeed, it is elementary to observe that  $e_{\alpha,\beta}(0;z) = 0$ when $\beta >1$, while $e_{\alpha,1}(0;z) = 1$ and $\lim_{t\to 0^{+}} e_{\alpha,\beta}(t;z) = + \infty$ for $\beta < 1$. 

In the following we will make use of an important result concerning the integration of $e_{\alpha,\beta}$ \cite{Podlubny1999}.

\begin{lem}\label{lem:IntegrationMittagLeffler}
Let $a < t$, $\Re(\alpha) > 0$, $\beta>0$ and $r\in\Rset$ such that $r>-1$. Then for any $z\in\Cset$
\begin{eqnarray*}
	\frac{1}{\Gamma(r+1)} \int_{a}^{t} e_{\alpha,\beta}(t-s;z) (s-a)^r ds
	= e_{\alpha,\beta+r+1}(t-a;z).
\end{eqnarray*}
\end{lem}

To keep the notation as compact as possible, for $a<b\le t$ and $r\in\Rset$ we conveniently introduce the function 
\begin{equation}\label{eq:DefinitionR}
	R_{\alpha,\beta,r}(t,a,b;z) = \frac{1}{\Gamma(r+1)} \int_{a}^{b} e_{\alpha,\beta}(t-s;z) (s-a)^r ds .
\end{equation}

The function $R_{\alpha,\beta,r}(t,a,b;z)$ is itself a generalization of $e_{\alpha,\beta}(t;z)$; indeed, Lemma \ref{lem:IntegrationMittagLeffler} implies that $R_{\alpha,\beta,r}(t,a,t;z) = e_{\alpha,\beta+r+1}(t-a;z)$.


In several situations it has been observed that it can be convenient to scale the function $e_{\alpha,\beta}$ according to 
\begin{equation}\label{eq:ScalingFunctionE}
	e_{\alpha,\beta}(t;z) = h^{\beta-1} e_{\alpha,\beta}(t/h;h^{\alpha}z) ,
\end{equation}
where $h$ is any positive real value \cite{GarrappaPopolizio2011_CAMWA,GarrappaPopolizio2011}. A similar scaling holds also for $R_{\alpha,\beta,k}(t,a,b;z)$ since, as it can be proved by using Lemma \ref{lem:DefinedIntegrationMittagLefflerScaled} in the Appendix, it is
\begin{equation}\label{eq:DefinedIntegrationMittagLeffler}
	R_{\alpha,\beta,k}(t,a,b;z) 
	= h^{\beta+k} R_{\alpha,\beta,k}\bigg(\frac{t-a}{h},0,1;h^{\alpha}z\bigg)
	, \quad h=b-a. 
\end{equation}

\section{Exponential quadrature rules of fractional order}\label{S:NumericalMethod}

Let us consider on $[t_{0},T]$ an equispaced mesh $t_{j}=t_{0}+jh$, with step--size $h>0$, and rewrite Equation (\ref{eq:FDE_VariationConstantFormula}) 
 in a piecewise manner as
\begin{equation}\label{eq:FDE_VariationConstantFormulaPiecewise}
	y(t_{n}) = \sum_{k=0}^{m-1}  e_{\alpha,k+1}(t_{n}-t_{0};\lambda) y_{0,k} 
	+ \sum_{j=0}^{n-1} \int_{t_{j}}^{t_{j+1}} e_{\alpha,\alpha}(t_{n}-s;\lambda) f(s) ds .
\end{equation}
  
To approximate the weighted integrals 
\[
	{\mathcal I}_j\bigl[f,t_n\bigr] = \int_{t_{j}}^{t_{j+1}} e_{\alpha,\alpha}(t_n-s;\lambda) f(s) ds 
\]
we fix $\nu$ nodes $c_{1},c_{2},\dots,c_{\nu} \in [0,1]$ and consider the quadrature rules of the type
\begin{equation}\label{eq:QuadratureRule}
	{\mathcal Q}_{n,j}\bigl[f\bigr] =
	\sum_{r=1}^{\nu} b_{r}(n-j) f(t_{j}+c_{r}h) 
\end{equation}
(obviously, weights $b_{r}(j)$ depend on $\alpha$ and $\lambda$ but for simplicity we omitted this dependence in the notation). 

The remainder of the approximation is
\[
	{\mathcal T}_{n,j}\bigl[f\bigr] = {\mathcal I}_j\bigl[f,t_{n}\bigr] - {\mathcal Q}_{n,j}\bigl[f\bigr] 
\]
which we can better analyze by extending to ${\mathcal Q}_{n,j}$ the Peano's theorem for classical quadrature rules (e.g., see \cite{DavisRabinowitz1984}). To this purpose it is useful to first consider, for any $k\in\Nset$ and $\xi\in\Rset$, the generalized step function
\[
	G_{\xi,k}(t)\equiv (t-\xi)_{+}^k = \left\{ 
		\begin{array}{ll}
		(t-\xi)^k & \textrm{if } \xi \le t \\
		0 & \textrm{if } \xi > t \\
		\end{array}. \right.
\]

\begin{thm}\label{thm:FDE_PeanoKernel1}
Let $f\in{\mathcal C}^{K}\bigl([t_{j},t_{j+1}]\bigr)$ and assume $f^{(K+1)}\in{\mathcal C}\bigl(]t_{j},t_{j+1}[\bigr)$. Then
\[
		{\mathcal T}_{n,j}\bigl[f\bigr] 
		= \sum_{k=0}^{K} \frac{f^{(k)}(t_{j})}{k!}  {\mathcal T}_{n,j}[G_{t_j,k}] 
		+ \int_{t_{j}}^{t_{j+1}} \frac{f^{(K+1)}(\xi)}{K!} {\mathcal T}_{n,j}[G_{\xi,K}] d\xi .
\]
\end{thm}

\begin{proof}
By replacing the function $f$ in $ {\mathcal I}_j\bigl[f,t_{n}\bigr]$ and $f(t_{j}+c_{r}h)$ in (\ref{eq:QuadratureRule}) respectively with their Taylor expansions
\[
	f(s) = 
	\sum_{k=0}^{K} \frac{f^{(k)}(t_{j})}{k!}(s-t_{j})^k + \int_{t_{j}}^{s} \frac{(s-\xi)^{K}}{K!} f^{(K+1)}(\xi) d\xi 
\]
\[
	f(t_{j}+c_{r}h) 
	= \sum_{k=0}^{K} \frac{f^{(k)}(t_{j})}{k!}h^k c_{r}^k
			+ \int_{t_{j}}^{t_{j}+c_{r}h} \frac{(t_{j}+c_{r}h-\xi)^{K}}{K!} f^{(K+1)}(\xi) d\xi 
\]
it is immediate, after subtraction, to see that 
\[
	{\mathcal T}_{n,j}\bigl[f\bigr] 
	= \sum_{k=0}^{K} \frac{f^{(k)}(t_{j})}{k!} {\mathcal T}_{n,j}\bigl[G_{t_j,k}\bigr] 
	+ \tilde{T}_{n,j}[f]
\]
where $\tilde{T}_{n,j}[f]$ is given by
\begin{eqnarray*}
	\tilde{T}_{n,j}[f]
	&=& \int_{t_{j}}^{t_{j+1}} e_{\alpha,\alpha}(t_{n}-s;\lambda)
			\int_{t_{j}}^{s} \frac{(s-\xi)^{K}}{K!} f^{(K+1)}(\xi) d\xi \,\,ds\\
	& & - \sum_{r=1}^{\nu} b_{r}(n-j) 
			 \int_{t_{j}}^{t_{j}+c_{r}h} \frac{(t_{j}+c_{r}h-\xi)^{K}}{K!} f^{(K+1)}(\xi) d\xi .
\end{eqnarray*}

To study $\tilde{T}_{n,j}[f]$ we preliminarily observe that by inverting the order of integration we are able to obtain
\[
	\int_{t_{j}}^{t_{j+1}} e_{\alpha,\alpha}(t_{n}-s;\lambda) \int_{t_{j}}^{s} \frac{(s-\xi)^{K}}{K!} f^{(K+1)}(\xi) d\xi  ds
	= \int_{t_{j}}^{t_{j+1}}  \frac{f^{(K+1)}(\xi)}{K!} {\mathcal I}_j[G_{\xi,K},t_{n}] d\xi 
\]
and, furthermore, since
\[
	\int_{t_{j}}^{t_{j}+c_{r}h} f^{(K+1)}(\xi) (t_{j}+c_{r}h-\xi)^{K} d\xi 
	= \int_{t_{j}}^{t_{j+1}} f^{(K+1)}(\xi) (t_{j}+c_{r}h-\xi)_{+}^{K} d\xi,
\]
it is elementary to see that 
\[
	\tilde{T}_{n,j}(f) =
	\int_{t_{j}}^{t_{j+1}} f^{(K+1)}(\xi) \frac{1}{K!}
	{\mathcal T}_{n,j}\bigl[G_{\xi,K}\bigr] d\xi
\]
from which the proof follows. 
\end{proof}

The classical concept of {\em{degree of precision}} applies to the quadrature ${\mathcal Q}_{n,j}$ as stated in the following definition.

\begin{defn}\label{defn:DegreePrecision}
A quadrature rule ${\mathcal Q}_{n,j}$ is said to have {\em{degree of precision}} $d$ if it is exact for any polynomial of degree not exceeding $d$ but it is not exact for polynomials of degree $d+1$.
\end{defn}

The following proposition shows that the function ${\mathcal T}_{n,j}\bigl[G_{\xi,K}\bigr] / K!$
is the $K$--th {\em{Peano kernel}} of the error of the quadrature ${\mathcal Q}_{n,j}$ if it has degree of precision $K$. It is an immediate consequence of Definition \ref{defn:DegreePrecision} and Theorem \ref{thm:FDE_PeanoKernel1} and therefore we omit the proof.

\begin{prop}\label{prop:ErrorQuadratureRulePeanoKernel}
Let $f\in{\mathcal C}^{K}\bigl([t_{j},t_{j+1}]\bigr)$ and assume $f^{(K+1)}\in{\mathcal C}\bigl(]t_{j},t_{j+1}[\bigr)$. If ${\mathcal Q}_{n,j}$ has degree of precision $K$ then
\begin{equation}\label{eq:FDE_ConvolutionQuadrature_ErrorPeanoKernel}
		{\mathcal T}_{n,j}\bigl[f\bigr] 
		= \int_{t_{j}}^{t_{j+1}} \frac{f^{(K+1)}(\xi)}{K!} {\mathcal T}_{n,j}\bigl[G_{\xi,K}\bigr] d\xi .
\end{equation}
\end{prop}

\begin{cor}\label{cor:ErrorPeanoConstantSign}
Let $f\in{\mathcal C}^{K}\bigl([t_{j},t_{j+1}]\bigr)$ and assume $f^{(K+1)}\in{\mathcal C}\bigl(]t_{j},t_{j+1}[\bigr)$. If ${\mathcal Q}_{n,j}$ has degree of precision $K$ and its Peano Kernel does not change its sign in $[t_{j},t_{j+1}]$, there exists a $\zeta \in ]t_j,t_{j+1}[$ such that
\[
		{\mathcal T}_{n,j}\bigl[f\bigr] = \frac{f^{(K+1)}(\zeta)}{(K+1)!} {\mathcal T}_{n,j}\bigl[G_{t_j,K+1}\bigr].
\]
\end{cor}
\begin{proof}
Since ${\mathcal T}_{n,j}\bigl[G_{\xi,K}\bigr]$ has the same sign in $[t_{j},t_{j+1}]$ the mean value theorem can be applied to (\ref{eq:FDE_ConvolutionQuadrature_ErrorPeanoKernel}) to obtain
\[
		{\mathcal T}_{n,j}\bigl[f\bigr] 
		= \frac{f^{(K+1)}(\zeta)}{K!} \int_{t_{j}}^{t_{j+1}}{\mathcal T}_{n,j}\bigl[G_{\xi,K}\bigr] d\xi ,
\]
with $t_{j}<\zeta<t_{j+1}$. The application of the above equation to $G_{t_j,K+1}$ leads to 
\[
	{\mathcal T}_{n,j}\bigl[G_{t_j,K+1}\bigr] 
		= K \int_{t_{j}}^{t_{j+1}} {\mathcal T}_{n,j}\bigl[G_{\xi,K}\bigr] d\xi 
\]
which allows to conclude the proof. 
\end{proof}

Under reasonable assumptions of smoothness for the forcing term $f(t)$, Proposition \ref{prop:ErrorQuadratureRulePeanoKernel} and its Corollary \ref{cor:ErrorPeanoConstantSign} allow to express the error of a quadrature rule (\ref{eq:QuadratureRule}) in terms of its degree of precision in a similar way as in the integer--order case.

The next step is to find necessary and sufficient conditions of algebraic type to characterize the degree of precision of the quadrature rules and hence derive the main coefficients. To this purpose we consider the following result.

\begin{prop}\label{prop:DegreePrecisionCharacterization}
A quadrature rule ${\mathcal Q}_{n,j}$ has degree of precision $d$ if and only if
\begin{equation}\label{eq:OrderConditions}
	\sum_{r=1}^{\nu} b_{r}(j) c_{r}^k =
		R_{\alpha,k}(j;h^{\alpha}\lambda) 
	, \quad k = 0, 1, \dots, d,
\end{equation}
where for shortness we put $R_{\alpha,k}(j;h^{\alpha}\lambda) = k! h^{\alpha} R_{\alpha,\alpha,k}(j,0,1;h^{\alpha}\lambda)$.
\end{prop}
\begin{proof}
By (\ref{eq:DefinitionR}) it is ${\mathcal I}_j[G_{t_j,k},t_{n}] = k! R_{\alpha,\alpha,k}(t_{n},t_{j},t_{j+1};\lambda)$ and, 
after using (\ref{eq:DefinedIntegrationMittagLeffler}), we have
\[
	{\mathcal T}_{n,j}\bigl[G_{t_j,k}\bigr] = h^{k}k! 
		\Bigl( h^{\alpha} R_{\alpha,\alpha,k}(n-j,0,1,h^{\alpha}\lambda) - \sum_{r=1}^{\nu} b_{r}(n-j) \frac{c_{r}^k}{k!} \Bigr)
\]
and the proof follows thanks to Theorem \ref{thm:FDE_PeanoKernel1}.
\end{proof}

\begin{rem}
Observe that when $\alpha=1$ and $\lambda=0$, ${\mathcal I}_j[f,t_{n}]$ becomes a standard (i.e. non--weighted) integral with ${\mathcal Q}_{n,j}$ the corresponding classical quadrature rule; on the other side, thanks to Lemma \ref{eq:LimitValuesR} in the Appendix, the order conditions (\ref{eq:OrderConditions}) become
\[
	\sum_{r=1}^{\nu} b_{r}(j) c_{r}^k = \frac{h}{k+1}
\]
which are the standard conditions for the degree of precision of classical rules (note that in this case, as expected, weights $b_{r}(j)$ are constant with respect to $j$).
\end{rem}

\section{Error analysis of exponential quadrature rules for FDEs}\label{S:ApplicationsFDS}

We can now investigate the behavior of quadrature rules ${\mathcal Q}_{n,j}$ when applied for solving the linear FDE (\ref{eq:FDE_Linear}). As described before, the idea is to resort to the expression  (\ref{eq:FDE_VariationConstantFormulaPiecewise}) for its solution; then, in each interval $[t_{j},t_{j+1}]$, one can think to apply the quadrature rule ${\mathcal Q}_{n,j}$ (\ref{eq:QuadratureRule}). The resulting numerical approximation $y_{n}$ of  $y(t_n)$ would read as the convolution quadrature

\begin{equation}\label{eq:FDE_ConvolutionQuadrature}
	y_{n} =
	\sum_{k=0}^{m-1}  e_{\alpha,k+1}(t_{n}-t_{0};\lambda) y_{0,k} + \sum_{j=0}^{n-1} \sum_{r=1}^{\nu} b_{r}(n-j) f(t_{j}+c_{r}h) .
\end{equation}

Thanks to the results on the error ${\mathcal T}_{n,j}\bigl[f\bigr]$ of each quadrature rule ${\mathcal Q}_{n,j}$ we are able to study the accuracy of the numerical scheme (\ref{eq:FDE_ConvolutionQuadrature}) by means of the following result.

\begin{prop}\label{prop:FDE_ErrorSolution}
Let $f\in{\mathcal C}^{d}\bigl([t_{0},T]\bigr)$ and assume that, for any $j=0,\ldots,n-1$, $f^{(d+1)}\in{\mathcal C}\bigl(]t_{j},t_{j+1}[\bigr)$ and ${\mathcal Q}_{n,j}$ has degree of precision $d$. Then 
\begin{eqnarray*} 
	\lefteqn{
	y(t_{n}) - y_{n} 
	= h^{d+1} \sum_{j=0}^{n-1} \int_0^1 f^{(d+1)}(t_{j}+uh) \cdot }  \hspace{1.0cm} && \\
	& & \hspace{0.3cm} \cdot
			\left( h^{\alpha} R_{\alpha,\alpha,d}(n-j,u,1;h^{\alpha} \lambda) - 
			\sum_{r=1}^{\nu} b_{r}(n-j) \frac{(c_{r}-u)_{+}^{d}}{d!}\right) du . \\
\end{eqnarray*}
\end{prop}
\begin{proof}
By subtracting (\ref{eq:FDE_ConvolutionQuadrature}) from (\ref{eq:FDE_VariationConstantFormulaPiecewise}) the error $y(t_{n})-y_{n}$ can be written as a sum, for $j=0,\dots,n-1$, of the errors ${\mathcal T}_{n,j}\bigl[f\bigr]$ of each quadrature ${\mathcal Q}_{n,j}$ and Proposition \ref{prop:ErrorQuadratureRulePeanoKernel} can be hence applied to each error term. Observe that for $\xi \in [t_{j},t_{j+1}]$ it is
\[
	{\mathcal I}_j\bigl[G_{\xi,d},t_{n}\bigr] 
	= \int_{\xi}^{t_{j+1}} e_{\alpha,\alpha}(t_{n}-s;\lambda) (s-\xi)^{d} ds 
\]
and hence, from Proposition \ref{lem:DefinedIntegrationMittagLefflerScaled},
\[
	{\mathcal I}_j\bigl[G_{\xi,d},t_{n}\bigr] = d! h^{\alpha+d} R_{\alpha,\alpha,d}(n-j,u,1;h^{\alpha} \lambda) ,
\]
where $u=(t_{j}-\xi)/h$. We  can therefore write the $d$--th Peano kernel of the quadrature ${\mathcal Q}_{n,j}$ as
\[
{\mathcal T}_{n,j}\bigl[G_{\xi,K}\bigr]
	= K! h^{d} \left( h^{\alpha} R_{\alpha,\alpha,d}(n-j,u,1;h^{\alpha} \lambda) - 
		\sum_{r=1}^{\nu} b_{r}(n-j) \frac{(c_{r}-u)_{+}^{d}}{d!} \right)
\]
from which the proof immediately follows. 
\end{proof}

Proposition \ref{prop:FDE_ErrorSolution} provides some useful information to relate the accuracy of the numerical method (\ref{eq:FDE_ConvolutionQuadrature}) with the degree of precision of the underlying quadrature rule. We recall that a numerical method like (\ref{eq:FDE_ConvolutionQuadrature}) converges with order $p$ if $y(t_{n})-y_{n}={\mathcal O}\bigl(h^{p}\bigr)$, as $h\to0$. Thus, if the quadrature rules ${\mathcal Q}_{n,j}$ have degree of precision $d$, the resulting method (\ref{eq:FDE_ConvolutionQuadrature}) converges with order $d+1$.

Anyway the analysis based only on the degree of precision and the Peano kernel is sometimes not adequate to disclose full information on the behavior of the quadrature (\ref{eq:FDE_ConvolutionQuadrature}).

To make a deeper analysis we focus on rules devised in an optimal way: in practice, if a degree of precision $K-1$ is required we use exactly $K$ nodes. Indeed, the $K$ algebraic conditions (\ref{eq:OrderConditions}) to get the degree of precision $K-1$ lead to a linear system and the existence and uniqueness of its solution are guaranteed only when the rule has exactly $K$ nodes and weights.

\begin{prop}\label{prop:ErrorInterpolatoryRule}
Let $f\in{\mathcal C}^{K+1}\bigl([t_{j},t_{j+1}]\bigr)$ and assume that $c_{1},\dots,c_{K}$ are $K$ distinct nodes. If ${\mathcal Q}_{n,j}$ has degree of precision $K-1$ then
\[
	{\mathcal T}_{n,j}\bigl[f\bigr] = \frac{h^{K+1}}{K!} \sum_{i=0}^{\infty} \frac{(t_{n}-t_{0})^{\alpha i + \alpha -1}\lambda^{i}}{\Gamma(\alpha i + \alpha)}
	\Omega_{n,j}^{[i,K]}[f]
	  + {\mathcal O}\bigl(h^{K+2}\bigr) ,
\]
where
\[
	\Omega_{n,j}^{[i,K]}[f] = \int_{0}^{1} \pi_{K}(u) \left(1-\frac{j}{n}-\frac{u}{n}\right)^{\alpha i + \alpha -1} f^{(K)}(t_{j}+uh) du
\]
and $\pi_{K}(u)=(u-c_{1})\cdots(u-c_{K})$.
\end{prop}
\begin{proof}
Consider the polynomial $p$, of degree $K-1$, interpolating the function $f$ at the nodes $t_{j}+c_{1}h,\dots,t_{j}+c_{K}h$, for which it is elementary to observe that 
\[
	{\mathcal I}_{j}[p,t_{n}] 
	= \sum_{r=1}^{K} \tilde{b}_{r}(n,j) f(t_{j}+c_{r}h)
	, \quad
	\tilde{b}_{r}(n,j) = \int_{t_{j}}^{t_{j+1}} e_{\alpha,\alpha}(t_{n}-s;\lambda) L_{r}\bigl((s-t_{j})/h\bigr) ds ,
\]
where $\bigl\{L_{r}\bigl(u\bigr)\bigr\}_{r=1,K}$ denotes the usual Lagrangian basis. ${\mathcal I}_{j}[p,t_{n}]$ provides a quadrature rule (of the same kind of (\ref{eq:QuadratureRule})) which is exact when $f(t)=(t-t_{j})^k$, $k=0,\dots,K-1$. Hence the weights $\tilde{b}_{r}(n,j)$ satisfy the order conditions (\ref{eq:OrderConditions}). Moreover since nodes $c_{r}$ are distinct, the system defined by (\ref{eq:OrderConditions}) has a unique solution. As a consequence $b_{r}(n-j) = \tilde{b}_{r}(n,j)$ and the quadrature rule ${\mathcal Q}_{n,j}$ coincides with the interpolatory rule ${\mathcal I}_{j}\bigl[p,t_{n}]$. Hence also the errors coincide; since for any $u\in[0,1]$ they can be written as
\[
	f(t_{j}+uh) - p(t_{j}+uh) = \frac{h^{K}\pi_{K}(u)}{K!} f^{(K)}(t_{j}+uh) + {\mathcal O}\bigl(h^{K+1}\bigr) ,
\]
then 
\[
	{\mathcal T}_{n,j}\bigl[f\bigr] = \frac{h^{K+1}}{K!} \int_{0}^{1} f^{(K)}(t_{j}+uh) e_{\alpha,\alpha}(t_{n}-t_{j}-uh;\lambda) \pi_{K}(u) du + {\mathcal O}\bigl(h^{K+2}\bigr) .
\]

Using the series expansion of the function $e_{\alpha,\alpha}$ leads to
\begin{eqnarray*} 
	{\mathcal T}_{n,j}\bigl[f\bigr] 
	&=& \frac{h^{K+1}}{K!} \sum_{i=0}^{\infty} \frac{(t_{n}-t_{0})^{\alpha i + \alpha -1}\lambda^{i}}{\Gamma(\alpha i + \alpha)} \cdot \\
	& & \hspace{0.3cm}
			\cdot \int_{0}^{1} \pi_{K}(u) \left(1-\frac{j}{n}-\frac{u}{n}\right)^{\alpha i + \alpha -1} f^{(K)}(t_{j}+uh) du 
			+ {\mathcal O}\bigl(h^{K+2}\bigr) \\
\end{eqnarray*}
from which the proof follows. 
\end{proof}

Thanks to the above result we are now able to represent the error of the convolution quadrature (\ref{eq:FDE_ConvolutionQuadrature}) in a more meaningful way.

\begin{prop}\label{prop:ErrorCompoundQudaratureRule}
Let $f\in{\mathcal C}^{K+1}\bigl([t_{0},T]\bigr)$ and assume that $c_{1},\dots,c_{K}$ are $K$ distinct nodes. If each ${\mathcal Q}_{n,j}$ has degree of precision $K-1$ then
\[
	y(t_{n}) - y_{n} = 
	\frac{h^{K}P_{K}}{K!} 
	\int_{t_{0}}^{t_{n}} e_{\alpha,\alpha}(t_{n}-s;\lambda) f^{(K)}(s) ds + {\mathcal O}\bigl(h^{K+\min\{\alpha,1\}}\bigr) 
\]
where $P_{K}=\int_{0}^{1} \pi_{K}(u) du$.
\end{prop}

\begin{proof}
The starting point of this proof is Proposition \ref{prop:ErrorInterpolatoryRule}. Observe that
\[
	\frac{1}{n} \sum_{j=0}^{n-1} \left(1-\frac{j}{n}-\frac{u}{n}\right)^{\alpha i + \alpha -1} f^{(K)}(t_{j}+uh)
\]
is the general trapezoidal rule $R^{[n]}[g;\gamma]$, $\gamma=2u-1$, when applied to the function $g(s) =(1-s)^{\alpha i + \alpha -1} f^{(K)}\bigl(t_{0}+s(t_{n}-t_{0})\bigr)$; this problem has been studied in \cite{LynessNinham1967} where it has been detected that, if $I[g] = \int_{0}^{1}  g(s) ds$, then 
\[
	E^{[n]}[g;\gamma] = I[g] - R^{[n]}[g;\gamma] = \left(u-\frac{1}{2}\right) f^{(K)}(t_{0}) n^{-1} 
	+ {\mathcal O}\bigl(n^{-\min\{2,\alpha(i+1)\}}\bigr).
\] 
Thus we have
\begin{eqnarray*}
	\frac{1}{n} \sum_{j=0}^{n-1} \Omega_{n,j}^{[i,K]}[f]
	&=& \int_{0}^{1} \pi_{K}(u) R^{[n]}[g;2u-1] du \\
	&=& \int_{0}^{1} \pi_{K}(u) I[g] du - \int_{0}^{1} \pi_{K}(u) E^{[n]}[g;2u-1] du \\
\end{eqnarray*}
and hence, after putting $\eta_{n}(s) = t_{0}+s(t_{n}-t_{0})$, it is 
\[
	\frac{1}{n} \sum_{j=0}^{n-1} \Omega_{n,j}^{[i,K]}[f] = 
		\int_{0}^{1} \pi_{K}(u) du \int_{0}^{1} (1-s)^{\alpha i + \alpha -1} f^{(K)}\bigl(\eta_{n}(s)\bigr)) ds 
	 	+ {\mathcal O}\bigl(n^{-\min\{1,\alpha(i+1)\}}\bigr).
\]
Therefore 
\begin{eqnarray*}
	\lefteqn{
	y(t_{n}) - y_{n} 
	 =  \frac{h^{K+1}n}{K!} P_{K} \sum_{i=0}^{\infty} \frac{(t_{n}-t_{0})^{\alpha i + \alpha -1}\lambda^{i}}{\Gamma(\alpha i + \alpha)}
			\int_{0}^{1} (1-s)^{\alpha i + \alpha -1} f^{(K)}\bigl(\eta_{n}(s)\bigr)) ds } \hspace{1.5cm} && \\
	& + & \frac{h^{K+1}n}{K!} \sum_{i=0}^{\infty} 
			\frac{(t_{n}-t_{0})^{\alpha i + \alpha -1}\lambda^{i}}{\Gamma(\alpha i + \alpha)} {\mathcal O}\bigl(n^{-\min\{1,\alpha(i+1)\}}\bigr)
	  	+ n{\mathcal O}\bigl(h^{K+2}\bigr) \\
\end{eqnarray*}
and 
\begin{eqnarray*} 
	y(t_{n}) - y_{n} 
	&=& \frac{h^{K+1}P_{K}n}{K!}  \sum_{i=0}^{\infty} \frac{(t_{n}-t_{0})^{\alpha i + \alpha -1}\lambda^{i}}{\Gamma(\alpha i + \alpha)} \cdot \\
	& & \hspace{0.3cm} \cdot
			\int_{0}^{1} (1-s)^{\alpha i + \alpha -1} f^{(K)}\bigl(\eta_{n}(s)\bigr)) ds + {\mathcal O}\bigl(h^{K+\min\{\alpha,1\}}\bigr) . \\
\end{eqnarray*}

The proof now follows after moving the summation in the integral and applying a simple change of variable.  
\end{proof}


Proposition \ref{prop:ErrorCompoundQudaratureRule} is of practical importance since it gives the explicit expression of the principal term of the error, thus allowing to monitor it and even to zero it out. In fact, the constant $P_{K}$ varies according to the selected nodes and can be set to zero by choosing them in a suitable way. It is elementary to see that when $K=1$ the choice of $c_{1}=1/2$ is the only one allowing $P_{1}=0$ whilst for $K=2$ any selection of nodes with $(c_{1}+c_{2})/2 - c_{1}c_{2} = 1/3$ gives $P_{2}=0$ (for instance, $c_{1}=0$, $c_{2}=2/3$ or $c_{1}=1/3$, $c_{2}=1$). Similarly for higher values of $K$. In general, with an odd number $K$ of nodes it suffices to choose the nodes symmetrically with respect to the midpoint of $[0,1]$.

As an immediate consequence of Proposition \ref{prop:ErrorCompoundQudaratureRule}, when the nodes are chosen such that $P_{K}=0$, the order of the quadrature rule increases from $K$ to $K+\min\{\alpha,1\}$; the order $K+1$ is hence obtained only when $\alpha \ge1$, while for $0<\alpha<1$ only a fraction $\alpha$ of order can be gained. 

For ease of presentation we summarize the above discussion in the following Corollary.

\begin{cor}\label{cor:OrderConvergence}
Let $f\in{\mathcal C}^{K+1}\bigl([t_{0},T]\bigr)$ and assume that $c_{1},\dots,c_{K}$ are $K$ distinct nodes chosen such that conditions (\ref{eq:OrderConditions}) hold for $d=K-1$. Then method (\ref{eq:FDE_ConvolutionQuadrature}) converges with order $K$. Moreover, for nodes chosen such that $P_{K}=0$, the order of convergence is $K+\min\{\alpha,1\}$.
\end{cor}



\section{Evaluation of quadrature weights}\label{S:EvaluationQuadratureWeights}

We now discuss some technical aspects concerning the evaluation of quadrature weights $b_{r}(j)$ in order to fulfill order conditions (\ref{eq:OrderConditions}).

From Proposition \ref{prop:DegreePrecisionCharacterization}, the weights vector ${\mathbf b}_{\nu}(j)=\bigl(b_{1}(j), b_{2}(j) , \dots, b_{\nu}(j)\bigr)^{T}$ must be solution of the linear system ${\mathbf C}_{\nu} {\mathbf b}_{\nu}(j) = {\mathbf R}_{\alpha,\nu}(j; h^{\alpha} \lambda)$, where   
\[
	{\mathbf C}_{\nu} = 
	\left( \begin{array}{ccc}
		1 & \dots & 1 \\
		c_{1} & \dots & c_{\nu} \\
		\vdots &  \ddots & \vdots \\
		c_{1}^{\nu-1} & \dots & c_{\nu}^{\nu-1} \\		
	\end{array} \right)
	, \quad 
		{\mathbf R}_{\alpha,\nu}(j;z) = 
		\left(\begin{array}{c}
			R_{\alpha,0}(j;z) \\ R_{\alpha,1}(j;z) \\ \vdots \\ R_{\alpha,\nu-1}(j;z) \\
	\end{array}\right)
	.
\]


The appearance of the mildly ill-conditioned Vandermonde matrix ${\mathbf C}_{\nu}$ does not represent a particular difficulty; indeed  very low dimensions are involved in practical implementations, since the number of quadrature nodes usually does not exceed $3$ or $4$. 

The most remarkable task concerns with the evaluation of the generalized ML functions $e_{\alpha,\beta}(j;z)$ in ${R}_{\alpha,k}(j;z)$ (see Lemma \ref{lem:DefinedIntegrationMittagLeffler}) which is a non--trivial task and can represent a bottleneck if not well organized.  Moreover, convolution quadrature (\ref{eq:FDE_ConvolutionQuadrature}) demands for the computation of a large number of ML functions and therefore fast algorithms should be provided in order to keep the whole computation at a reasonable level.


To this purpose, methods based on the numerical inversion of the Laplace transform can be of help. Recently there has been a renewed interest in this approach (see, for instance, \cite{LopezPalenciaSchadle2006,WeidemanTrefethen2007,GarrappaPopolizio2013}) which turns out to be suitable and highly--competitive for evaluating functions, like the ML ones, whose Laplace transform has a simpler representation than the function itself. 

Indeed, it is a well known result \cite{Mainardi2010} that the Laplace transform of $e_{\alpha,\beta}(t;z)$ is  
\[
	{\mathcal L} \bigl(e_{\alpha,\beta}(t;z);s\bigr) =
	\frac{s^{\alpha-\beta}}{s^{\alpha}+z}
	, \quad \Re(s) > |z|^{\frac{1}{\alpha}} ,
\]
with a branch cut along the negative real semi-axis to uniquely define the real powers of $s$.

After putting $\tau=sj$, the integral representation of $e_{\alpha,\beta}(j;z)$ can be written as 
\begin{equation}\label{eq:MLInverseLaplaceTransform}
	e_{\alpha,\beta}(j;z) = 
	\frac{j^{\beta-1}}{2\pi i} \int_{\mathcal C} e^{\tau} \frac{\tau^{\alpha-\beta}}{\tau^{\alpha} + j^{\alpha} z} d\tau ,
\end{equation}
where the  contour ${\mathcal C}$ represents a deformation of the  Bromwich line which begins and ends in the left half--plane and encircles the circular disk $|\tau| = j|z|^{\frac{1}{\alpha}}$ in the positive sense.

Among several possible approaches to numerically evaluate formulas like (\ref{eq:MLInverseLaplaceTransform}) we consider here the approach described in \cite{TrefethenWeidemanSchmelzer2006} (and already exploited in the numerical solution of linear FDEs in \cite{GarrappaPopolizio2011}) which, broadly speaking, consists in replacing the exponential function $e^{\tau}$ by a $(N-1,N)$ rational approximation ${\mathcal R}_{N}(\tau) = \sum_{k=1}^{N} {r_{k}}/{(\tau-\tau_{k})}$. By assuming the contour ${\mathcal C}$ winding in the negative sense once around each pole $\tau_{k}$, the application of the Cauchy's integral formula in  (\ref{eq:MLInverseLaplaceTransform}) with $e^{\tau}$ replaced by ${\mathcal R}_{N}(\tau)$, together with Lemma \ref{lem:DefinedIntegrationMittagLeffler}, leads to the approximation for $R_{\alpha,r}(j;z)$ in the form

\begin{eqnarray}\label{eq:DiscreteInverseLaplaceTransform}
	R_{\alpha,r}^{[N]}(j;z) = r! \left( 
		- j^{\alpha+r} \sum_{k=1}^{N} \frac{r_{k}  \tau_{k}^{-r-1}}{\tau_{k}^{\alpha} + j^\alpha z}
		+ \sum_{l=0}^{r-1} \frac{(j-1)^{\alpha+l}}{(r-l)!} \sum_{k=1}^{N} \frac{r_{k} \tau_{k}^{-l-1}}{\tau_{k}^{\alpha}+(j-1)^{\alpha}z}
		\right) .
\end{eqnarray}

The computation of the above formula is not very expensive since the powers $\tau_{k}^{\alpha}$ and $\tau_{k}^{l}$, $l=0,1,\dots,r$, can be computed once and reused at each time--level $j$; furthermore poles and residues occur in complex conjugate pairs thus allowing, for real $\lambda$, to almost halve the computational cost.

The approximation to $R_{\alpha,r}(j;z)$ provided by $R_{\alpha,r}^{[N]}(j;z)$ strictly depends on the approximation ${\mathcal R}_{N}(\tau)$ used for $e^{\tau}$. It seems quite natural to focus on Chebyshev rational functions (whose residues and poles are readily available thanks to the work of Cody et al. \cite{CodyMeinardusVarga1969}), which represent the best rational approximation to the exponential on $(-\infty,0)$; moreover their error decays with a rate proportional to $9.28903^{-N}$ \cite{PetrushevPopov1987}, thus allowing to reach high precision, nearly close to the round-off error, with a reasonable degree $N$. 


\subsection{Multidimensional systems}\label{S:Multidimensional}

The extension of the results discussed in the previous sections to the multidimensional case, i.e. when $\lambda$ is a matrix and $f$ a vector--valued function, is straightforward. First observe that the matrix function $e_{\alpha,\beta}(t;h^{\alpha}\lambda)$ is defined, in a similar way as for scalar instances, by inverting its Laplace transform
\[
	{\mathcal L} \bigl(e_{\alpha,\beta}(t;z);s\bigr) =
	s^{\alpha-\beta} \bigl( s^{\alpha}I + h^{\alpha} \lambda\bigr)^{-1} ,
\]
where $I$ is the identity matrix of the same size of $\lambda$ and $s$ satisfies $\Re(s) > |\lambda_{i}|^{\frac{1}{\alpha}}$ with $\lambda_{i}$ denoting the eigenvalues of $\lambda$. For simplicity we will assume that the spectrum of $\lambda$ lies on the positive real semi--axis. In a similar way the definition of $R_{\alpha,r}(j;h^{\alpha}\lambda)$ can be extended to matrix arguments.

The multidimensional case thus involves the numerical evaluation of matrix functions; this topic receives great attention since matrix functions are useful tools not only in applied mathematics and scientific computing, but also in various other fields like control theory and complex networks. A wide literature is devoted to this topic; we just cite, among the others, 
\cite{DelBuonoLopez2002,MoretNovati2004,LopezSimoncini2006BIT,DelBuonoLopezPeluso2006SIAM,PoliPo,PopolizioSimoncini2008,MoretPopolizio2012} and the thorough book by Higham \cite{Higham1}. 

The weights of the quadrature rule (\ref{eq:FDE_ConvolutionQuadrature}) are now some matrices, say $B_{r}(j)$, which are evaluated again by means of (\ref{eq:OrderConditions}) as the solution of the system $\bigl( {\mathbf C}_{\nu} \otimes I \bigr)  {\mathbf B}(j) = h^{\alpha} {\mathbf R}_{\alpha,\nu}(j;h^{\alpha} \lambda)$, with ${\mathbf B}(j) = \bigl( B_{1}(j)^T, B_{2}(j)^T, \dots, B_{\nu}(j)^T \bigr)^{T}$ and $\otimes$ the Kronecker product. Note that this is a set of linear systems, each for any column in ${\mathbf B}(j)$ or ${\mathbf R}_{\alpha,\nu}(j;h^{\alpha} \lambda)$. In these cases once a suitable factorization of the coefficient matrix is computed, it can be applied to all systems.  In our situation, thanks to the Kronecker structure, only a factorization of ${\mathbf C}_{\nu}$ is needed, as in the scalar case.

The matrix counterpart of (\ref{eq:DiscreteInverseLaplaceTransform}) for the numerical approximation of $R_{\alpha,k}(j;h^{\alpha}\lambda)$ involves the inversion of some matrices. These evaluations, obviously, lead to an increase of the computational effort necessary for the integration of the FDE (\ref{eq:FDE_Linear}). By means of some numerical experiments, we will observe that this issue does not impair the advantages in accuracy of the proposed approach.

\section{Numerical experiments}\label{S:NumericalExperiments}

In this Section we present numerical experiments to validate theoretical results and compare the approach under investigation with some other schemes for FDEs. As a first case, we consider the test problem (\ref{eq:FDE_Linear}) with the forcing term
\begin{equation}\label{eq:ForcingTermTest1}
	f(t) = \frac{1}{\Gamma(p+1-\alpha)} (t-t_{0})^{p-\alpha}
	, \quad p > \left\lceil \alpha \right\rceil-1  ,
\end{equation}
where the parameter $p$ will be selected, for each experiment, according to the smoothness requirements as investigated in Section \ref{S:NumericalMethod} and \ref{S:ApplicationsFDS}, and depending on the nodes number $\nu$.

When coupled with the initial condition $y(t_{0})=y_{0}$ for $0<\alpha<1$ and with the set of initial conditions $y(t_{0})=y_{0}$ and $y'(t_{0})=0$ for $1<\alpha<2$, the exact solution of this problem is $y(t) = E_{\alpha,1}\bigl(-(t-t_0)^{\alpha}\lambda\bigr) y_0 + (t-t_0)^{p-1}E_{\alpha,p}\bigl(-(t-t_0)^{\alpha}\lambda\bigr)$ and can be evaluated, with accuracy up to round--off error, by means of the Matlab {\tt mlf} function \cite{PodlubnyKacenak2009} with the very small tolerance of $10^{-15}$. A rational approximation with degree $N=15$ is used in (\ref{eq:DiscreteInverseLaplaceTransform}) with the aim of calculating weights with a negligible error.

In the following experiments we use $\lambda=3$, $t_0=0$ and we evaluate the numerical solution at $T=1$ by a sequence of decreasing step--sizes $h$. Errors $E(h)$, with respect to the exact solution, are reported together with an estimate of the order of convergence, denoted with EOC and obtained as $\log_{2}\bigl(E(h)/E(h/2)\bigr)$. 

All the experiments are performed using Matlab ver. 7.7 on a computer running the 32-bit Windows Vista Operating System and equipped with the Intel quad-core CPU Q9400 at 2.66GHz.

For the first group of experiments we consider methods (\ref{eq:FDE_ConvolutionQuadrature}) with $\nu=1$ and quadrature weights derived from order conditions (\ref{eq:OrderConditions}) with $d=0$. We solve the test problem with different values of the single node $c_{1}$ to point out the different behavior according to the choice of the node. In the optimal case (i.e., $c_{1}=\frac{1}{2}$), since the factor $P_1=\frac{1}{2}-c_{1}$ in the principal part of the error vanishes, we expect from Proposition \ref{prop:ErrorCompoundQudaratureRule} an order of convergence equal to $1+\min\{\alpha,1\}$ while in the remaining cases order $1$ is expected.

Results in Tables \ref{tab:Exp1_Problem1_alpha50} and \ref{tab:Exp1_Problem1_alpha150}, respectively for $\alpha=0.5$ and $\alpha=1.5$, agree with the theoretical findings and confirm the importance of selecting nodes in a proper way to improve the convergence properties.

\begin{table}[htb!]
\scriptsize
\[
   \begin{array}{|r|cc|cc|cc|} \hline
     & \multicolumn{2}{|c|}{\textrm{     $c_1=0$     }} & \multicolumn{2}{|c|}{\textrm{$c_1=\frac{1}{2}$}} & \multicolumn{2}{|c|}{\textrm{     $c_2=1$     }} \\ 
   h & \textrm{Error} & \textrm{EOC} & \textrm{Error} & \textrm{EOC} & \textrm{Error} & \textrm{EOC} \\ \hline
1/4 & 5.26(-2) &   & 1.98(-2) &   & 1.59(-2) &   \\
1/8 & 2.53(-2) & 1.058 & 8.08(-3) & 1.293 & 9.81(-3) & 0.693 \\
1/16 & 1.19(-2) & 1.082 & 3.17(-3) & 1.349 & 5.77(-3) & 0.766 \\
1/32 & 5.63(-3) & 1.084 & 1.21(-3) & 1.390 & 3.25(-3) & 0.826 \\
1/64 & 2.67(-3) & 1.076 & 4.52(-4) & 1.421 & 1.78(-3) & 0.872 \\
1/128 & 1.28(-3) & 1.064 & 1.66(-4) & 1.443 & 9.48(-4) & 0.907 \\
\hline
   \end{array}
\]
\normalsize
\caption{Errors and EOC for $\nu=1$ with forcing term (\ref{eq:ForcingTermTest1}), $\alpha=0.5$ and $p=2.0$.}
\label{tab:Exp1_Problem1_alpha50}
\end{table}

\begin{table}[htb!]
\scriptsize
\[
   \begin{array}{|r|cc|cc|cc|} \hline
     & \multicolumn{2}{|c|}{\textrm{     $c_1=0$     }} & \multicolumn{2}{|c|}{\textrm{$c_1=\frac{1}{2}$}} & \multicolumn{2}{|c|}{\textrm{     $c_2=1$     }} \\ 
   h & \textrm{Error} & \textrm{EOC} & \textrm{Error} & \textrm{EOC} & \textrm{Error} & \textrm{EOC} \\ \hline
1/4 & 3.47(-2) &   & 3.55(-4) &   & 4.12(-2) &   \\
1/8 & 1.81(-2) & 0.936 & 1.45(-4) & 1.294 & 1.99(-2) & 1.050 \\
1/16 & 9.28(-3) & 0.966 & 4.49(-5) & 1.690 & 9.74(-3) & 1.030 \\
1/32 & 4.70(-3) & 0.983 & 1.27(-5) & 1.826 & 4.82(-3) & 1.016 \\
1/64 & 2.36(-3) & 0.991 & 3.41(-6) & 1.892 & 2.39(-3) & 1.009 \\
1/128 & 1.19(-3) & 0.995 & 8.96(-7) & 1.929 & 1.19(-3) & 1.005 \\
\hline
   \end{array}
\]
\normalsize
\caption{Errors and EOC for $\nu=1$ with forcing term (\ref{eq:ForcingTermTest1}), $\alpha=1.5$ and $p=3.0$.}
\label{tab:Exp1_Problem1_alpha150}
\end{table}

A similar experiment is carried out for methods with $\nu=2$ constructed after imposing order conditions (\ref{eq:OrderConditions}) for $d=1$; we first consider the $2$ nodes $c=\{0, \, 1\}$, for which an order of convergence equal to 2 is predicted, and hence we consider two combinations of optimal nodes, namely $c=\{0, \, 2/3\}$ and $c=\{1/3, \, 1\}$ for which the expected order of convergence is $2+\min\{1,\alpha\}$. Also these experiments (see Table \ref{tab:Exp2_Problem1_alpha50} for $\alpha=0.5$ and Table \ref{tab:Exp2_Problem1_alpha150} for $\alpha=1.5$) validate in a quite clear manner the results of Section \ref{S:ApplicationsFDS}.   

\begin{table}[htb!]
\scriptsize
\[
   \begin{array}{|r|cc|cc|cc|} \hline
     & \multicolumn{2}{|c|}{\textrm{$c=\{0,1\}$}} & \multicolumn{2}{|c|}{\textrm{$c=\{0,\frac{2}{3}\}$}} & \multicolumn{2}{|c|}{\textrm{$c=\{\frac{1}{3},1\}$}} \\ 
   h & \textrm{Error} & \textrm{EOC} & \textrm{Error} & \textrm{EOC} & \textrm{Error} & \textrm{EOC} \\ \hline
1/4 & 8.92(-4) &   & 1.35(-3) &   & 2.63(-4) &   \\
1/8 & 2.61(-4) & 1.775 & 2.72(-4) & 2.310 & 6.07(-5) & 2.113 \\
1/16 & 7.25(-5) & 1.844 & 5.25(-5) & 2.371 & 1.31(-5) & 2.213 \\
1/32 & 1.95(-5) & 1.892 & 9.88(-6) & 2.412 & 2.68(-6) & 2.289 \\
1/64 & 5.15(-6) & 1.925 & 1.82(-6) & 2.439 & 5.26(-7) & 2.347 \\
1/128 & 1.33(-6) & 1.948 & 3.31(-7) & 2.458 & 1.00(-7) & 2.390 \\
\hline
   \end{array}
\]
\normalsize
\caption{Errors and EOC for $\nu=2$ with forcing term (\ref{eq:ForcingTermTest1}), $\alpha=0.5$ and $p=3.0$.}
\label{tab:Exp2_Problem1_alpha50}
\end{table}

\begin{table}[htb!]
\scriptsize
\[
   \begin{array}{|r|cc|cc|cc|} \hline
     & \multicolumn{2}{|c|}{\textrm{$c=\{0,1\}$}} & \multicolumn{2}{|c|}{\textrm{$c=\{0,\frac{2}{3}\}$}} & \multicolumn{2}{|c|}{\textrm{$c=\{\frac{1}{3},1\}$}} \\ 
   h & \textrm{Error} & \textrm{EOC} & \textrm{Error} & \textrm{EOC} & \textrm{Error} & \textrm{EOC} \\ \hline
1/4 & 1.61(-3) &   & 5.05(-5) &   & 3.99(-6) &   \\
1/8 & 3.99(-4) & 2.014 & 6.70(-6) & 2.914 & 2.63(-6) & 0.600 \\
1/16 & 9.93(-5) & 2.005 & 8.57(-7) & 2.966 & 5.02(-7) & 2.393 \\
1/32 & 2.48(-5) & 2.002 & 1.08(-7) & 2.983 & 7.72(-8) & 2.699 \\
1/64 & 6.20(-6) & 2.001 & 1.37(-8) & 2.990 & 1.09(-8) & 2.824 \\
1/128 & 1.55(-6) & 2.000 & 1.71(-9) & 2.994 & 1.47(-9) & 2.888 \\
\hline
   \end{array}
\]
\normalsize
\caption{Errors and EOC for $\nu=2$ with forcing term (\ref{eq:ForcingTermTest1}), $\alpha=1.5$ and $p=4.0$.}
\label{tab:Exp2_Problem1_alpha150}
\end{table}

The same experiment is again repeated for $\nu=3$ and $d=2$ and the results are presented in Table \ref{tab:Exp3_Problem1_alpha50}. For brevity, in this case we omit to present the case $\alpha >1$.

 
\begin{table}[htb!]
\scriptsize
\[
   \begin{array}{|r|cc|cc|cc|} \hline
     & \multicolumn{2}{|c|}{\textrm{$c=\{0, \, 0.8, \, 1\}    $}} & \multicolumn{2}{|c|}{\textrm{$c=\{0, \, 0.5, \, 1\}    $}} & \multicolumn{2}{|c|}{\textrm{$c=\{0.2, \, 0.5, \, 0.8\}$}} \\ 
   h & \textrm{Error} & \textrm{EOC} & \textrm{Error} & \textrm{EOC} & \textrm{Error} & \textrm{EOC} \\ \hline
1/4 & 1.27(-5) &   & 3.29(-5) &   & 2.26(-5) &   \\
1/8 & 2.17(-6) & 2.545 & 3.24(-6) & 3.342 & 2.23(-6) & 3.339 \\
1/16 & 3.39(-7) & 2.679 & 3.08(-7) & 3.396 & 2.13(-7) & 3.390 \\
1/32 & 4.96(-8) & 2.773 & 2.86(-8) & 3.430 & 1.98(-8) & 3.423 \\
1/64 & 6.93(-9) & 2.839 & 2.61(-9) & 3.453 & 1.82(-9) & 3.446 \\
1/128 & 9.37(-10) & 2.886 & 2.36(-10) & 3.468 & 1.65(-10) & 3.462 \\
\hline
   \end{array}
\]
\normalsize
\caption{Errors and EOC for $\nu=3$ with forcing term (\ref{eq:ForcingTermTest1}), $\alpha=0.5$ and $p=4.0$.}
\label{tab:Exp3_Problem1_alpha50}
\end{table}

In Table \ref{tab:Exp4_Problem1_alpha50} we summarize errors and EOCs, for $\alpha=0.5$, obtained with an increasing number $\nu$ of nodes chosen in the optimal way described at the end of Section \ref{S:ApplicationsFDS}, and we compare them with the results provided by the fractional PECE method \cite{DiethelmFreed1999}. We have chosen a forcing term smooth enough to assure high accuracy for all the selected methods (in this case $p=6$).

The improvement in the accuracy, with respect to the fractional PECE method, is remarkable for all the methods and becomes striking as $\nu$ increases. Note that even with $\nu=4$ an accuracy very close to the machine precision is reached; thus simulations with $\nu > 4$ are not carried out.

\begin{table}[htb!]
\scriptsize
\[
   \begin{array}{|r|cc|cc|cc|cc|cc|} \hline
     & \multicolumn{2}{|c|}{\textrm{$c=\{\frac{1}{2}\}$}} & \multicolumn{2}{|c|}{\textrm{$c=\{\frac{1}{3},1\}$}} & \multicolumn{2}{|c|}{\textrm{$c=\{0,\frac{1}{2},1\}$}} & \multicolumn{2}{|c|}{\textrm{$c=\{0,\frac{1}{4},\frac{7}{10},1\}$}} & \multicolumn{2}{|c|}{\textrm{Frac. PECE}} \\ 
   h & \textrm{Error} & \textrm{EOC} & \textrm{Error} & \textrm{EOC} & \textrm{Error} & \textrm{EOC} & \textrm{Error} & \textrm{EOC} & \textrm{Error} & \textrm{EOC} \\ \hline
1/4 & 2.54(-4) &   & 1.58(-5) &   & 2.08(-6) &   & 7.59(-8) &   & 2.28(0) &   \\
1/8 & 1.18(-4) & 1.106 & 4.02(-6) & 1.976 & 2.59(-7) & 3.007 & 4.20(-9) & 4.175 & 4.84(-2) & 5.554 \\
1/16 & 4.95(-5) & 1.253 & 9.14(-7) & 2.137 & 2.87(-8) & 3.175 & 2.13(-10) & 4.301 & 4.33(-3) & 3.482 \\
1/32 & 1.95(-5) & 1.340 & 1.93(-7) & 2.245 & 2.95(-9) & 3.281 & 1.02(-11) & 4.378 & 1.03(-3) & 2.070 \\
1/64 & 7.44(-6) & 1.394 & 3.86(-8) & 2.320 & 2.89(-10) & 3.350 & 4.63(-13) & 4.466 & 2.98(-4) & 1.795 \\
1/128 & 2.76(-6) & 1.428 & 7.46(-9) & 2.372 & 2.75(-11) & 3.395 & 8.91(-15) & 5.701 & 9.31(-5) & 1.676 \\
\hline
   \end{array}
\]
\normalsize
\caption{Errors and EOC for various methods with forcing term (\ref{eq:ForcingTermTest1}), $\alpha=0.5$ and $p=6.0$.}
\label{tab:Exp4_Problem1_alpha50}
\end{table}

The methods investigated in this paper are computationally more expensive with respect to other approaches; however from these tests we can see that this extra amount of computation allows to reach high accuracies, which are reachable by other approaches only with a very small step--size and hence a huge computational cost. Thus in Figure \ref{fig:ErrorCostsP1}, we plot errors against computational time from the previous experiment (execution time are the mean values over a very large number of executions).


\begin{figure}[htb!]
\centering
	\includegraphics[width=0.50\textwidth]{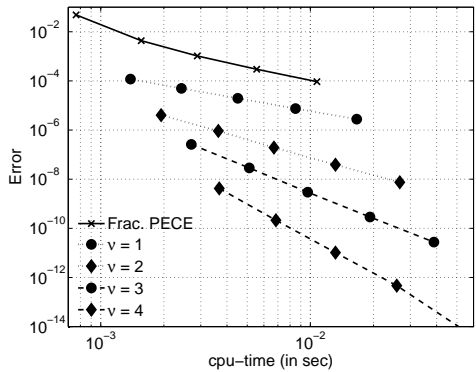} 
\caption{Accuracy versus time for experiments in Table \ref{tab:Exp4_Problem1_alpha50}}
\label{fig:ErrorCostsP1}
\end{figure}

As the plot clearly shows the execution time of the methods under investigation are reasonable and they do not increase exceedingly as the number $\nu$ of weights and nodes increases. Furthermore any given level of accuracy is obtained with the proposed scheme with a computational cost which is lower with respect to the classical fractional PECE method (for instance, an accuracy of $10^{-4}$ is obtained by the fractional PECE method in $10^{-2}$ CPU seconds and in approximatively $10^{-3}$ CPU seconds by the new method with $\nu=1$ or less when $\nu>1$).

In our next experiment the comparison is performed with a slightly more general forcing term 
\begin{equation}\label{eq:ForcingTermTest2}
	f(t) = \sin(t-t_{0}) + 3\cos(t-t_{0}).
\end{equation} 

As reference solution we consider that obtained by the method (\ref{eq:FDE_ConvolutionQuadrature}) with $\nu=4$ and a very small stepsize. Results in Table \ref{tab:Exp4_Problem2_alpha50} and Figure \ref{fig:ErrorCostsP2} show the effectiveness of the proposed method also for this test problem. 

\begin{table}[htb!]
\scriptsize
\[
   \begin{array}{|r|cc|cc|cc|cc|cc|} \hline
     & \multicolumn{2}{|c|}{\textrm{$c=\{\frac{1}{2}\}$}} & \multicolumn{2}{|c|}{\textrm{$c=\{\frac{1}{3},1\}$}} & \multicolumn{2}{|c|}{\textrm{$c=\{0,\frac{1}{2},1\}$}} & \multicolumn{2}{|c|}{\textrm{$c=\{0,\frac{1}{4},\frac{7}{10},1\}$}} & \multicolumn{2}{|c|}{\textrm{Frac. PECE}} \\ 
   h & \textrm{Error} & \textrm{EOC} & \textrm{Error} & \textrm{EOC} & \textrm{Error} & \textrm{EOC} & \textrm{Error} & \textrm{EOC} & \textrm{Error} & \textrm{EOC} \\ \hline
1/4 & 3.28(-2) &   & 6.36(-4) &   & 1.55(-5) &   & 3.19(-7) &   & 9.33(-2) &   \\
1/8 & 1.38(-2) & 1.253 & 1.41(-4) & 2.170 & 1.83(-6) & 3.084 & 1.57(-8) & 4.347 & 3.44(-2) & 1.440 \\
1/16 & 5.48(-3) & 1.329 & 2.98(-5) & 2.246 & 1.97(-7) & 3.217 & 7.49(-10) & 4.389 & 6.90(-3) & 2.317 \\
1/32 & 2.11(-3) & 1.380 & 6.01(-6) & 2.310 & 1.99(-8) & 3.304 & 3.49(-11) & 4.423 & 1.75(-3) & 1.976 \\
1/64 & 7.89(-4) & 1.415 & 1.17(-6) & 2.360 & 1.94(-9) & 3.363 & 1.60(-12) & 4.449 & 5.09(-4) & 1.786 \\
1/128 & 2.91(-4) & 1.440 & 2.22(-7) & 2.398 & 1.83(-10) & 3.406 & 7.26(-14) & 4.461 & 1.58(-4) & 1.683 \\
\hline
   \end{array}
\]
\normalsize
\caption{Errors and EOC for various methods with forcing term (\ref{eq:ForcingTermTest2}) and $\alpha=0.5$.}
\label{tab:Exp4_Problem2_alpha50}
\end{table}

\begin{figure}[htb!]
\centering
	\includegraphics[width=0.50\textwidth]{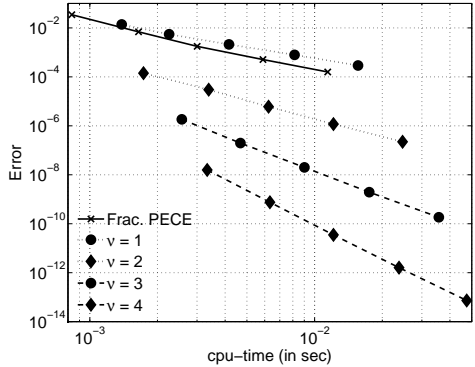} 
\caption{Accuracy versus time for experiments in Table \ref{tab:Exp4_Problem2_alpha50}}
\label{fig:ErrorCostsP2}
\end{figure}

Our last experiment concerns with a multidimensional system 
\begin{equation}\label{eq:Test3}
	D^{\alpha}_0 U(t) + A U(t) = F(t)
\end{equation}
where $A$ is an $M \times M$ matrix and $U$ and $F$ are $M$-valued functions. Problems of this kind occur frequently in applications; for instance, a time--fractional partial differential equation (PDE)
\begin{equation}\label{eq:Test3FractionalPDE}
	\frac{\partial^\alpha }{\partial t^{\alpha}}u(x,t) - \frac{\partial^2 u}{\partial x^2}u(x,t) = \frac{t^{p}}{\Gamma(p+1)} \sin \pi x 
\end{equation}
for $t>0$, $x\in[0,1]$ and subject to the initial and boundary conditions
\[
	u(x,0) = \sin \pi x, \quad
	u(0,t) = u(1,t) = 0 
\]
can be transformed into the system of FDEs (\ref{eq:Test3}) by the so--called \emph{Method of Lines}: once a mesh grid $x_{j} = j\Delta x$, $j=0,1,\dots,M+1$, $\Delta x = 1/(M+1)$, has been fixed on $[0,1]$, the spatial second derivative is approximated by means of the centered finite differences; in this way 
problem (\ref{eq:Test3FractionalPDE}) can be rewritten as (\ref{eq:Test3}) where $A$ is the tridiagonal matrix with entries $(-1,2,-1)$, scaled by the factor $(\Delta x)^2$, and
$U(t) = \bigl( u(x_{1},t), u(x_{2},t), \dots, u(x_{M},t) \bigr)^T, \,\, F(t) = \frac{t^{p}}{\Gamma(p+1)}\bigl( \sin (\pi x_{1}), \sin (\pi x_{2}), \dots, \sin (\pi x_{M}) \bigr)^T\,\in \Rset^{M}$.

More generally, $F(t)$ will include not only data arising from inhomogeneous terms in the differential operator but also possible inhomogeneous boundary data.

It is elementary to verify that the exact solution of the fractional PDE (\ref{eq:Test3FractionalPDE}) is given by $u(t,x) = \sin (\pi x) \left( e_{\alpha,1}(t;\pi^2) + e_{\alpha,p+\alpha+1}(t;\pi^2) \right)$ and its graph is plotted in Figure \ref{fig:Fig_FPDE_SolTeor}.

\begin{figure}[htb!]
\centering
	\includegraphics[width=0.50\textwidth]{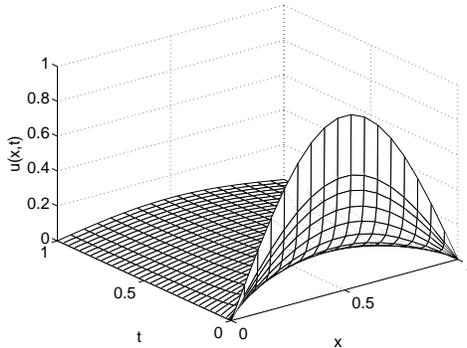} 
\caption{Exact solution of the fractional PDE (\ref{eq:Test3FractionalPDE}) for $\alpha=0.5$ and $p=3$.}
\label{fig:Fig_FPDE_SolTeor}
\end{figure}

In Table \ref{tab:FPDE_Problem3_N8_alpha80} we present the result of the computation for $\alpha=0.8$, $p=3$ and $M=8$. We note that for large values of $h$ the fractional PECE method fails in most cases to provide correct results due to instability problems. Indeed, the spectrum of the matrix $A$ lies in the interval $[-4/(\Delta x)^2,0]$ and, since the method is actually explicit, in order to obtain a stable behavior a very strong restriction on the step--size $h$, of the kind $h^{\alpha} < (\Delta x)^2 C_{\alpha} /4$, with $C_{\alpha} \in [1,2]$ (see \cite{Garrappa2010}), has to be imposed. Thus, to make a meaningful comparison we have considered also the implicit PI trapezoidal rule (the rule used as the predictor in the PECE algorithm) which possesses better stability properties. 

Also in this case we observe that method (\ref{eq:FDE_ConvolutionQuadrature}) allows to reach a greater accuracy with respect to the PI rule. Anyway the joined comparison of errors and computational costs, as showed in the first plot of Figure \ref{fig:ErrorCosts_FPDE_1}, emphasizes how the proposed method is more competitive especially when high accuracy is required. 

\begin{table}[htb!]
\scriptsize
\[
   \begin{array}{|r|cc|cc|cc|cc|} \hline
     & \multicolumn{2}{|c|}{\textrm{$\nu=2, \, c=\{\frac{1}{3},1\}$}} & \multicolumn{2}{|c|}{\textrm{$\nu=3, \, c=\{0,\frac{1}{2},1\}$}} & \multicolumn{2}{|c|}{\textrm{Frac. PECE}} & \multicolumn{2}{|c|}{\textrm{PI Trapez}} \\ 
   h & \textrm{Error} & \textrm{EOC} & \textrm{Error} & \textrm{EOC} & \textrm{Error} & \textrm{EOC} & \textrm{Error} & \textrm{EOC} \\ \hline
1/8 & 2.78(-5) &   & 4.97(-7) &   & 1.41(10) &   & 8.34(-4) &   \\
1/16 & 5.16(-6) & 2.429 & 4.21(-8) & 3.559 & 1.20(29) & * & 2.36(-4) & 1.822  \\
1/32 & 8.54(-7) & 2.594 & 3.30(-9) & 3.676 & 9.55(58) & * & 6.64(-5) & 1.829  \\
1/64 & 1.33(-7) & 2.686 & 2.47(-10) & 3.736 & 2.00(102) & * & 1.87(-5) & 1.827  \\
1/128 & 1.99(-8) & 2.735 & 1.82(-11) & 3.766 & 7.98(154) & * & 5.29(-6) & 1.823  \\
1/256 & 2.94(-9) & 2.762 & 1.32(-12) & 3.782 & 4.67(187) & * & 1.50(-6) & 1.820  \\
1/512 & 4.29(-10) & 2.777 & 9.56(-14) & 3.790 & 2.08(103) & * & 4.25(-7) & 1.817  \\
1/1024 & 6.22(-11) & 2.786 & 2.12(-14) & 2.175 & 5.59(-7) & * & 1.21(-7) & 1.815  \\
\hline
   \end{array}
\]
\normalsize
\caption{Errors and EOC for test problem (\ref{eq:Test3}) for $\alpha=0.8$, $N=8$ and $p=3.0$.}
\label{tab:FPDE_Problem3_N8_alpha80}
\end{table}


\begin{figure}[htb!]
\[
\begin{array}{cc}
	\includegraphics[width=0.50\textwidth]{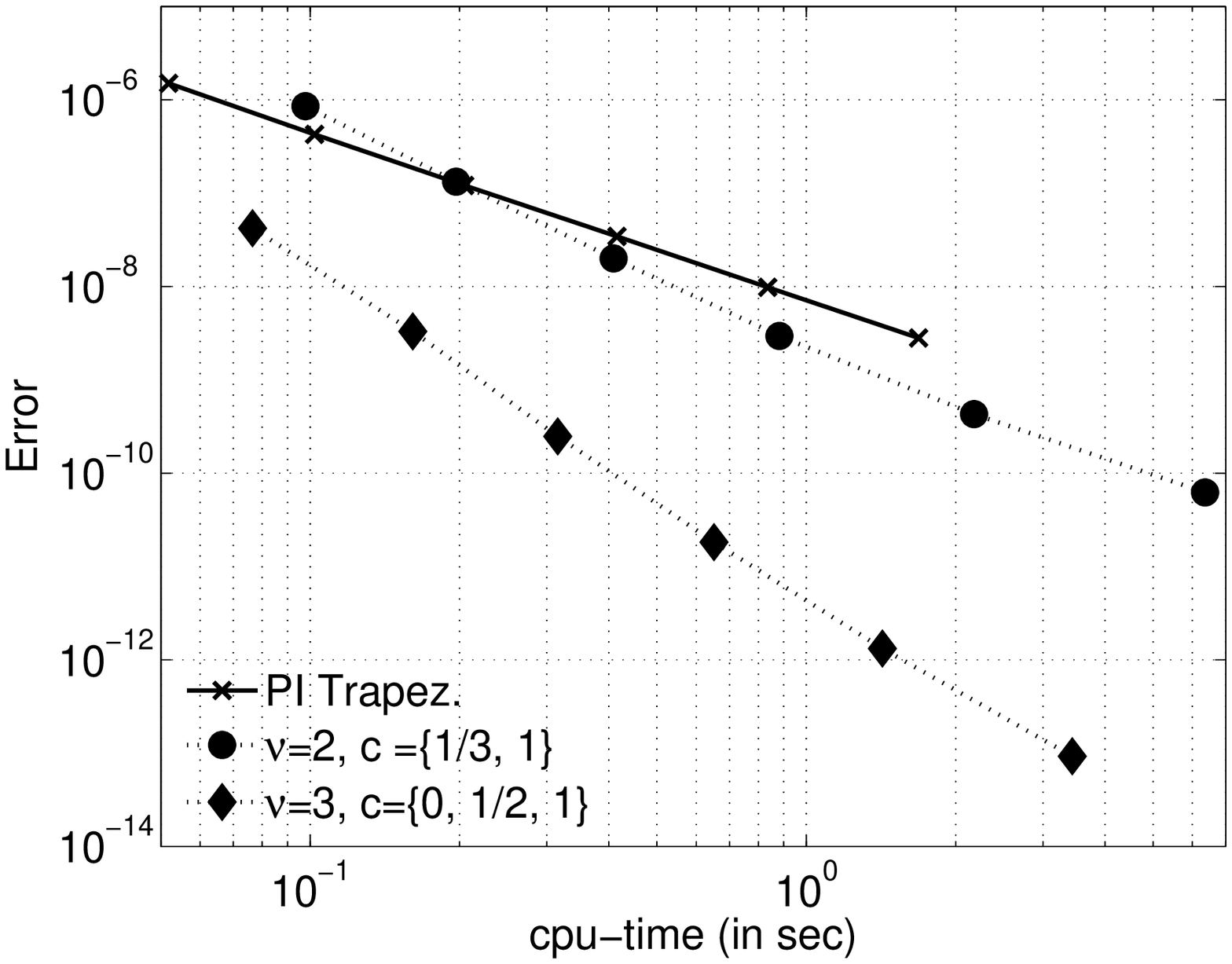} &
	\includegraphics[width=0.50\textwidth]{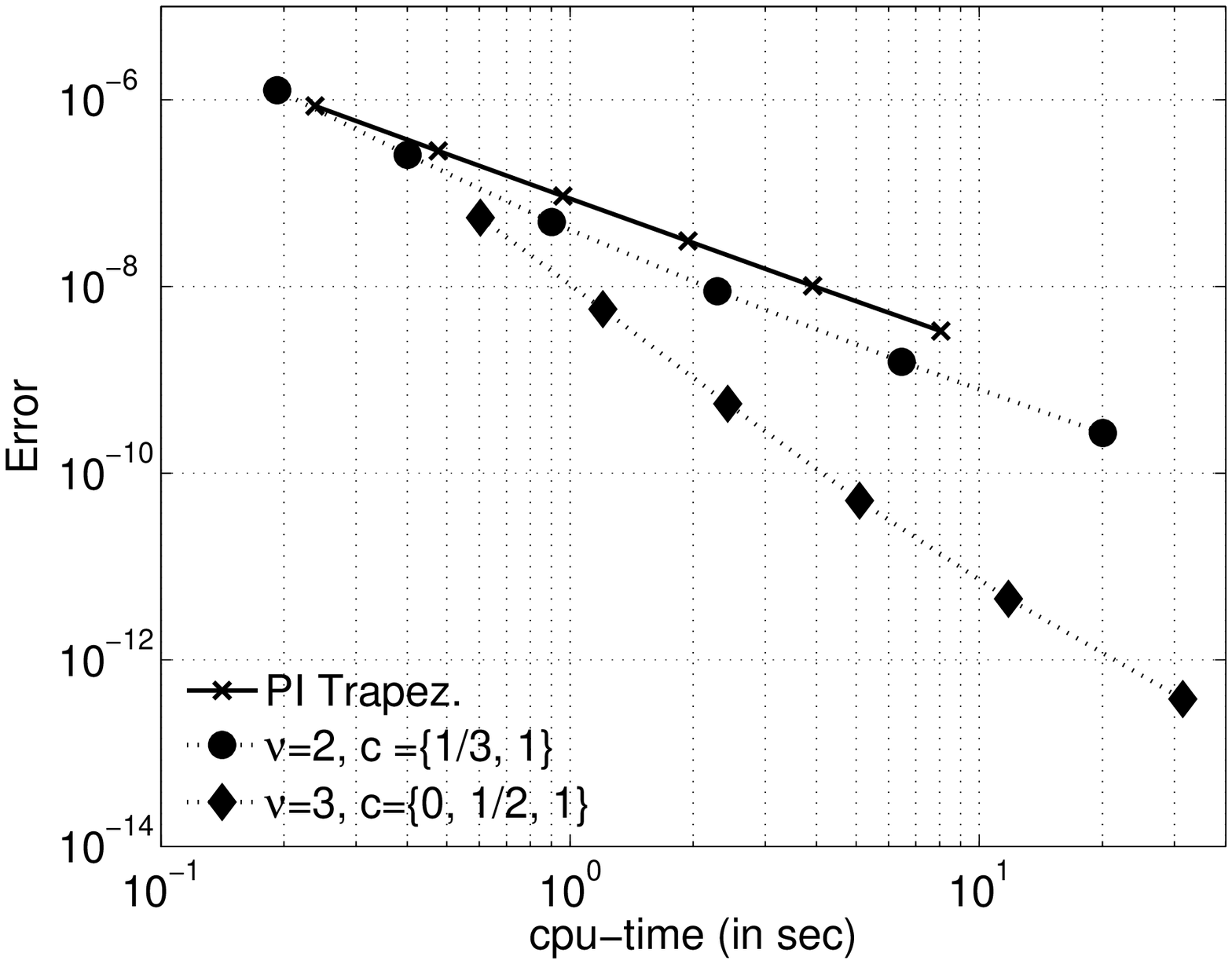} \\
\end{array}
\]	 
\caption{Accuracy versus time for for test problem (\ref{eq:Test3}) with $\alpha=0.9$, $M=8$ and $p=3.0$ (left plot) and $\alpha=0.6$, $M=16$ and $p=3.0$ (right plot).}
\label{fig:ErrorCosts_FPDE_1}
\end{figure}

Similar results are obtained for $M=16$ and $\alpha=0.6$ as shown in Table \ref{tab:FPDE_Problem3_N16_alpha60} and in the second plot of Figure \ref{fig:ErrorCosts_FPDE_1}.

\begin{table}[htb!]
\scriptsize
\[
   \begin{array}{|r|cc|cc|cc|} \hline
     & \multicolumn{2}{|c|}{\textrm{$\nu=2, \, c=\{\frac{1}{3},1\}$}} & \multicolumn{2}{|c|}{\textrm{$\nu=3, \, c=\{0,\frac{1}{2},1\}$}} & \multicolumn{2}{|c|}{\textrm{PI Trapez}} \\ 
   h & \textrm{Error} & \textrm{EOC} & \textrm{Error} & \textrm{EOC} & \textrm{Error} & \textrm{EOC} \\ \hline
1/8 & 2.30(-5) &   & 4.76(-7) &   & 2.17(-3) &   \\
1/16 & 5.68(-6) & 2.014 & 5.46(-8) & 3.125 & 7.00(-4) & 1.631 \\
1/32 & 1.26(-6) & 2.170 & 5.70(-9) & 3.259 & 2.27(-4) & 1.624 \\
1/64 & 2.57(-7) & 2.298 & 5.53(-10) & 3.366 & 7.41(-5) & 1.616 \\
1/128 & 4.88(-8) & 2.394 & 5.08(-11) & 3.444 & 2.43(-5) & 1.611 \\
1/256 & 8.87(-9) & 2.461 & 4.50(-12) & 3.497 & 7.96(-6) & 1.608 \\
1/512 & 1.56(-9) & 2.508 & 3.83(-13) & 3.554 & 2.61(-6) & 1.606 \\
1/1024 & 2.68(-10) & 2.539 & 3.28(-14) & 3.545 & 8.60(-7) & 1.604 \\
\hline
   \end{array}
\]
\normalsize
\caption{Errors and EOC for test problem (\ref{eq:Test3}) for $\alpha=0.6$, $N=16$ and $p=3.0$.}
\label{tab:FPDE_Problem3_N16_alpha60}
\end{table}


\section{Appendix}

\begin{lem}\label{lem:DefinedIntegrationMittagLeffler}
Let $a<b \le t$, $\Re(\alpha) > 0$, $\beta>0$ and $k\in\Nset$. Then 
\[
	R_{\alpha,\beta,k}(t,a,b;z)
	= e_{\alpha,\beta+k+1}(t-a;z) - \sum_{\ell=0}^{k} \frac{(b-a)^{k-\ell}}{(k-\ell)!} e_{\alpha,\beta+\ell+1}(t-b;z) .
\]
\end{lem}

\begin{proof}
First write
\[
	\int_{a}^{b} e_{\alpha,\beta}(t-s;z) (s-a)^k ds =
	\int_{a}^{t} e_{\alpha,\beta}(t-s;z) (s-a)^k ds 
	 - 	\int_{b}^{t} e_{\alpha,\beta}(t-s;z) (s-a)^k ds
\]
and hence, since for $h=b-a$ it is
\[
	(s-a)^k = (s-b+h)^k = \sum_{\ell=0}^{k} \binom{k}{\ell} h^{k-\ell} (s-b)^{\ell} ,
\]
the application of Lemma \ref{lem:IntegrationMittagLeffler} leads to the desired result
since  $\Gamma(k+1)=k!$, being $k\in\Nset$.  
\end{proof}



\begin{lem}\label{eq:LimitValuesR}
Let $a<b\le t$ and $k \in \Nset$. Then 
$R_{1,1,k}(t,a,b;0) = (b-a)^{k+1}/(k+1)!$.
\end{lem}
\begin{proof}
It is elementary to see that $
	e_{\alpha,\beta}(t;0) = t^{\beta-1}/\Gamma(\beta)
$
and hence Lemma \ref{lem:DefinedIntegrationMittagLeffler} leads to
\[
	R_{1,1,k}(t,a,b;0)
	= \frac{(t-a)^{k+1}}{\Gamma(k+2)} 
			- \sum_{\ell=0}^{k} \frac{h^{k-\ell}}{(k-\ell)!} \frac{(t-b)^{\ell+1}}{\Gamma(\ell+2)} ,
\]	
where $h=b-a$. By recalling that $\Gamma(n+1)=n!$ and
\[
	\frac{1}{(k-\ell)!(\ell+1)!} = \frac{1}{(k+1)!} \binom{k+1}{\ell+1} ,
\]
we can write
\begin{eqnarray*}
	R_{1,1,k}(t,a,b;0)
	&=& \frac{1}{(k+1)!} \left( (t-a)^{k+1} - \sum_{\ell=1}^{k+1} \binom{k+1}{\ell} h^{k-\ell+1}(t-b)^{\ell} \right) \\
	&=& \frac{1}{(k+1)!} \left( (t-a)^{k+1} - (t-b+h)^{k+1} + h^{k+1} \right)
\end{eqnarray*}	
from which the proof easily follows since $t-b+h=t-a$. 
\end{proof}

\begin{lem}\label{lem:DefinedIntegrationMittagLefflerScaled}
Let $a<b\le t$, $h=b-a$ and $k\in\Nset$. Then for any $\xi \in [a,b]$ it is
\[
	\frac{1}{\Gamma(k+1)} \int_{\xi}^{b} e_{\alpha,\beta}(t-s;z)(s-\xi)^k ds 
	= h^{\beta+k} R_{\alpha,\beta,k}\Bigl(\frac{t-a}{h},\frac{\xi-a}{h},1;h^{\alpha}z\Bigr)
\]
\end{lem}
\begin{proof}
First consider $u=(\xi-a)/h$. Thus
\[
	\int_{\xi}^{b} e_{\alpha,\beta}(t-s;z)(s-\xi)^k ds 
	= \int_{a+uh}^{b} e_{\alpha,\beta}(t-s;z)(s-a-uh)^k ds 
\]
and the change of variable $s=a+hr$ leads to
\[
	\int_{\xi}^{b} e_{\alpha,\beta}(t-s;z)(s-\xi)^k ds 
	= h^{k+1} \int_{u}^{1} e_{\alpha,\beta}(t-a-rh;z)(r-u)^k dr .
\]

By using (\ref{eq:ScalingFunctionE}) we have $e_{\alpha,\beta}(t-a-rh;z) = h^{\beta-1} e_{\alpha,\beta}((t-a)/h-r;h^{\alpha}z)$ from which the proof follows in a straightforward way. 
\end{proof}

\end{document}